\documentclass[12pt]{amsart}
\usepackage[utf8]{inputenc}
\usepackage{amsfonts}
\usepackage{amsmath,amsthm,amssymb}
\usepackage{bbold}

\theoremstyle{plain}

\newtheorem{theorem}{Theorem}

\newtheorem{lemma}[theorem]{Lemma}

\newtheorem{claim}{Claim}

\theoremstyle{remark}
\newtheorem*{Remark}{Remark}

\newcommand\supp{\operatorname{supp}}
\newcommand{\norm}[1]{\left\lVert#1\right\rVert}

\begin{document}

\title[]{On the existence of an extremal function in the Delsarte extremal problem}

\author{Marcell Ga\'al and Zsuzsanna Nagy-Csiha}

%\dedicatory{Dedicated to the memory of Professor D\'enes Petz}

\address{Marcell Ga\'al
\newline  \indent R\'enyi Institute of Mathematics \newline \indent Hungarian Academy of Sciences,
\newline \indent Budapest, Re\'altanoda utca 13-15,  1053 Hungary}
\email{gaal.marcell@renyi.hu}

\address{Zsuzsanna Nagy-Csiha
\newline \indent Department of Numerical Analysis 
\newline \indent Faculty of Informatics, Eötvös Loránd University,
\newline \indent Budapest, Pázmány Péter sétány 1/C, 1117 Hungary
\newline \indent and
\newline \indent Institute of Mathematics and Informatics 
\newline \indent Faculty of Sciences, University of Pécs,
\newline \indent Pécs, Ifjúság útja 6, 7624 Hungary}
\email{ncszsu@gamma.ttk.pte.hu}


\maketitle
\begin{abstract}
This paper is concerned with a Delsarte type extremal problem.
Denote by $\mathcal{P}(G)$ the set of positive definite continuous functions on a locally compact abelian group $G$. We consider the function class, which was originally introduced by Gorbachev,
\begin{multline*}
\mathcal{G}(W, Q)_G =
\left\{ f \in \mathcal{P}(G) \cap L^1(G) ~ :
\right.
\\ ~
\left. f(0) = 1, ~ \supp{f_+} \subseteq W,~  \supp \widehat{f} \subseteq Q \right\}
\end{multline*}
where $W\subseteq G$ is closed and of finite Haar measure and $Q\subseteq \widehat{G}$ is compact. We also consider the related Delsarte type problem of finding the extremal quantity
\begin{equation*}
\mathcal{D}(W,Q)_G = \sup \left\{ \int_{G} f(g) d\lambda_G(g) ~ : ~ f \in \mathcal{G}(W,Q)_G\right\}.
\end{equation*}
The main objective of the current paper is to prove the existence of an extremal function for the Delsarte type extremal problem $\mathcal{D}(W,Q)_G$. 
The existence of the extremal function has recently been established by Berdysheva and Révész in the most immediate case where $G=\mathbb{R}^d$.
So the novelty here is that we consider the problem in the general setting of locally compact abelian groups.
In this way our result provides a far reaching generalization of the former work of Berdysheva and Révész.

\end{abstract}

\bigskip
\bigskip

\section{Introduction}

The Fourier analytic formulation of the so-called Delsarte extremal problem on $\mathbb{R}^d$ incorporates the calculation of the numerical quantity
$$ \sup \widehat{f}(0) = \sup \frac{1}{(2\pi)^{\frac{d}{2}}} \int_{\mathbb{R}^{d}} f(x)dx $$ provided that
(i) $f\in L^1(\mathbb{R}^d)$, $f$ is continuous and bounded on $\mathbb{R}^d$, 
(ii) $f(0)=1$,
(iii) $f(x) \leq 0$ for $\|x\| \geq 2$ and
(iv) $\widehat{f}(y) \geq 0$.
The last property (iv) can be interpreted as $f$ being positive definite, see the precise definition of positive definiteness in the forthcoming section.

The Delsarte extremal problem has generated broad interest because of its intimate connections to different problems from various branches of mathematics.
First of all, the linear programming bound of Delsarte is useful in coding and design theory as well. Second, let us mention that relying on Delsarte's problem, upper bounds can be derived for the sphere packing density of $\mathbb{R}^d$ \cite{Cohn,Gorbachev,Levenshtein,Yudin}.
Moreover, Gorbachev and Tikhonov \cite{GT} worked out a further concrete application of the Delsarte problem for the so-called Wiener problem.

A few of years ago Viazovska \cite{Viazovska} solved the sphere packing problem in dimension 8, combining the Delsarte extremal problem with modular form techniques. Subsequently, in the paper \cite{CKMRV} Cohn et al. resolved the problem also in dimension 24.

Beside solving the Delsarte problem, further challenging and closely related questions come into picture. As for recent investigations in this direction, we refer to the seminal paper of Berdysheva and R\'ev\'esz \cite{Berdysheva}.
They have pointed out the independence of the extremal constant from the underlying function class. Furthermore, they showed the existence of an extremal function in band-limited cases.
The main objective of the current paper is to prove an analogous result for general LCA groups.

\section{The result}

Before moving on, we need some more preliminaries. In the first part of the section, we summarize the necessary background from the field of abstract harmonic analysis.
Let $G$ be a locally compact abelian group (LCA group for short).
The \emph{dual group} of $G$ is denoted by $\widehat{G}$, by which we mean the set of continuous homomorphisms of $G$ into the complex unit circle $\mathbb{T}$, the multiplication being the pointwise multiplication of functions.
For a compact set $K\subseteq G$ and an open set $U\subseteq \mathbb{T}$, consider the set
\[
P(K,U):=\{ \chi \in \widehat{G} ~ : ~ \chi(K) \subseteq U  \}.
\]
Then the \emph{compact open topology} on $\widehat{G}$ contains the sets $P(K,U)$ as a subbasis. By this topology, $\widehat{G}$ acquires an LCA group structure.
%As for the group operations, we will use additive notation for $G$ while we employ multiplicative notation for $\widehat{G}$, if not stated otherwise.
The \emph{Pontryagin-van Kampen Duality Theorem} asserts that $G$ is isomorphic to $\widehat{\widehat{G}}$, both as groups and as topological spaces.
In this case $\delta$ stands for the corresponding natural isomorphism, that is
\[
\delta_g(\chi):=\chi(g), \qquad \chi  \in \widehat{G}
\]
and
$\delta: G \to \widehat{\widehat{G}}$, $g\mapsto \delta_g$
which is usually called the \emph{Pontryagin map}.

Recall that a \emph{continuous} function $f \in C(G)$ is called \emph{positive definite} (denoted by $f \gg 0$) if the inequality
\begin{equation}\label{posdef} 
\sum_{j=1}^{n}\sum_{k=1}^{n} c_j \overline{c_k} f(g_j-g_k)\ge 0
\end{equation}
holds for all choices of $n\in \mathbb{N}$, $c_j\in \mathbb{C}$ and $g_j\in G$
for $j=1,\dots,n$.
Throughout the paper the set of continuous positive definite functions defined on $G$ will be denoted by $\mathcal{P}(G)$.
If $\lambda_G$ is a (fixed, conveniently normalized) Haar measure on $G$, then the condition
\eqref{posdef} for continuous $f$ is equivalent to 
    \[
    \int_G\int_G f(g-s) \varphi(g)\overline{\varphi(s)} d\lambda_G(g) d\lambda_G(s) \geq 0
    \]
    for every $\varphi \in L^1(G)$ (see, for instance \cite[13.4.4. Proposition]{Dixmier}).
The next properties will be quite useful in the sequel.
We have \cite[(32.4)]{HewittRossII}

\begin{lemma} \label{L:positivedefprop}
Let $G$ be an LCA group and denote by $\star$ the convolution.
\begin{enumerate}
    \item If $f$ is a positive definite function on $G$, then
        \begin{itemize}
            \item[a)] $\left|f(g)\right| \leq f(0)=\|f\|_{\infty}$ for all $g\in G$;
            \item[b)] $\int_G f d\lambda_G\geq 0$.
        \end{itemize}
    \item If $\varphi\in L^2(G)$ and $\widetilde{\varphi}$ is defined as $\widetilde{\varphi}(g):=\overline{\varphi(-g)}$ $(g\in G)$, then the convolution square $\varphi\star\widetilde{\varphi}$ is a continuous positive definite function.
\end{enumerate}
\end{lemma}

For any $f\in L^1(G)$, its \emph{Fourier transform} $\widehat{f}$ is defined on $\widehat{G}$ as
\[
\widehat{f}(\chi)=\int_G f(g) \overline{\chi(g)} d\lambda_G(g), \quad \chi \in \widehat{G}.
\]
The \emph{Inversion Theorem} (cf. \cite[Section 1.5]{Rudin}) asserts that if $f$ belongs to $[\mathcal{P}(G) \cap L^1(G)]$, the subspace generated by $\mathcal{P}(G) \cap L^1(G)$, then $\widehat{f}\in L^1\left(\widehat{G}\right)$ and the Haar measure $\lambda_{\widehat{G}}$ on $\widehat{G}$ can be normalized so that $f(g)=\widehat{\widehat{f}}\left(\delta_{g^{-1}}\right)$.
We shall use this Haar measure (the so-called Plancherel measure) on the dual group $\widehat{G}$.
For $k\in L^1\left(\widehat{G}\right)$, introducing the \emph{conjugate Fourier transform} $\mathcal{F}^*$ as
\[
\mathcal{F}^*(k)(g):= \int_{\widehat{G}} k(\chi) \delta_g(\chi) d\lambda_{\widehat{G}}(\chi), \quad g\in G,
\]
the Inversion Theorem can be rephrased as $f=\mathcal{F}^*\left(\widehat{f}\right)$ is satisfied for every $f \in [\mathcal{P}(G) \cap L^1(G)]$.
%The latter condition is satisfied automatically when $f\in L^1(G) \cap \mathcal{P}_G$. ***REF. Rudin ?***

%The \emph{Inversion Theorem} (cf. \cite[Section 1.5]{Rudin}) asserts that the Haar measure $\lambda_{\widehat{G}}$ on $\widehat{G}$ can be normalized so that
%$f=\mathcal{F}^*\left(\widehat{f}\right)$ is satisfied for all $f\in L^1(G)$ with $\widehat{f} \in L^1(\widehat{G})$.
%In the sequel we will use this normalized Haar measure on $\widehat{G}$.

Another important tool in our study is the \emph{Plancherel Theorem} which asserts that the Fourier transform $\mathcal{F}: [\mathcal{P}(G) \cap L^1(G)] \to L^2\left(\widehat{G}\right)$ can be extended to a unitary equivalence $U:L^2(G) \to L^2\left(\widehat{G}\right)$. This unitary operator is called the \emph{Plancherel transform}. We abuse notation and do not distinguish the usual Fourier transform and the latter extension.

Denote, as usual, $x_{+}:=\max(x,0)$ and $x_{-}:=\max(-x,0)$ for any $x\in \mathbb{R}$, with similar notation for functions as well. In this paper we consider the function class
\begin{multline}\label{G(W,Q)}
\mathcal{G}(W, Q)_G =
\left\{ f \in \mathcal{P}(G) \cap L^1(G) ~ :
\right.
\\ ~
\left. f(0) = 1, ~ \supp{f_+} \subseteq W,~  \supp \widehat{f} \subseteq Q \right\}
\end{multline}
where $W\subseteq G$ is closed and of finite Haar measure and $Q\subseteq \widehat{G}$ is compact.   
It was originally introduced by Gorbachev \cite{Gorbachev} in connection with the Delsarte type problem of finding the extremal quantity
\begin{equation} \label{D(W,Q)}
\mathcal{D}(W,Q)_G = \sup \left\{ \int_{G} f(g) d\lambda_G(g) ~ : ~ f \in \mathcal{G}(W,Q)_G\right\}
\end{equation}
in the most immediate case where $G=\mathbb{R}^d$, $$W=\mathbb{B}=\{x\in \mathbb{R}^d ~:~ \|x\| \leq 1 \}$$ and  $Q=r\mathbb{B}$ with some real number $r>0$.

In the very recent publication
\cite{Berdysheva}, Berdysheva and Révész
analysed in details the aforementioned Delsarte type extremal quantity.
When $G={\mathbb{R}}^d$, they collect
and work up extensive information, which were in part either folklore or just available in different unpublished sources, to clarify
existence of extremal functions
$f\in \mathcal{G}(W, Q)_{\mathbb{R}^d}$ in certain band-limited cases, that is, when $W$ is closed and of finite Lebesgue measure and $Q$ is compact.
%%%% with the property $$\int_{\mathbb{R}^d} f(x) dx = \mathcal{D}(W,Q).$$
Since the problem of existence of the extremal function makes sense also in case of general LCA groups, our objective is to obtain a completely analogous counterpart of the aforementioned result in the general setting of LCA groups. More precisely, we intend to prove the following.

\begin{theorem} \label{T:main}
Let $G$ be any LCA group. If $W \subseteq G$ is closed with positive, finite Haar measure and $Q \subseteq \widehat{G}$ is compact, then there exists an extremal function $f\in \mathcal{G}(W,Q)_G$ satisfying $\int_{G} f d\lambda_G = \mathcal{D}(W,Q)_G$.
\end{theorem}

Note that the existence of an extremal function might be helpful in calculating or estimating the extremal constant itself. That explains the effort undertaken, for instance in \cite{Berdysheva, CG, CKMRV, GSS, Gorbachev, GI, GT, I1, I2} to prove existence of extremal functions. As we will see, the argument of \cite{Berdysheva} cannot be directly copied here, because the estimation of modulus of smoothness, direct use and decrease estimates of Bessel functions, or $\sigma$-compactness of the underlying group $G$ are no longer available.

\section{Proof}

Recall that a Banach space $X$ is called \emph{weakly compactly generated} (WCG for short) if it has a weakly compact subset whose linear span is dense in $X$. Fundamental examples of such spaces are separable normed spaces and reflexive Banach spaces.

An important property what we shall apply in our argument is that the unit ball of the dual space of a WCG space is weak-$^*$ sequentially compact (see \cite{Diestel}, page 148).

For a general LCA group $G$, it might be difficult to characterize when $L^1(G)$ turns to be a WCG space, however,
a sufficient condition for that is the $\sigma$-compactness of $G$. This sufficiency can be seen by composing two well-known results. First note that the space $L^1(X,\mu)$ is WCG when the occurring measure $\mu$ is $\sigma$-finite on $X$ (see \cite{Phelps}, page 36). Second the Haar measure on the LCA group $G$ is $\sigma$-finite exactly when $G$ is $\sigma$-compact.
Moreover, in that case we have the duality $\left(L^1(G)\right)^*=L^{\infty}(G)$ of Banach spaces (see \cite[(20.20.) Theorem]{HS} and cf. \cite{Szekelyhidi}, page 11) because a $\sigma$-finite measure is decomposable.
%in the sense that all the continuous linear functionals on $L^1(G)$ admit the form
%\[
%\varphi \to \int_G \varphi g
%\]
%with some $g\in L^{\infty}(G)$.

%We remark that when the space $L^1(X,\mu)$ is WCG but $\mu$ is not $\sigma$-finite, then $\mu$ behaves rather pathological, for instance $\mu$ could be any measure which takes only the values $0$ and $\infty$.

The proof of Theorem \ref{T:main} rests heavily on a technical lemma.

\begin{lemma}\label{L:precompact}
Let $W\subseteq G$ be closed and of finite Haar measure and let $Q\subseteq \widehat{G}$ be compact.
Then the function class $\mathcal{G}(W,Q)_G\subseteq C(G)$ is relatively compact in the compact convergence topology.
\end{lemma}

In the setting $G=\mathbb{R}^d$ the above lemma has been a part of the proof of \cite[Proposition 3.5]{Berdysheva}, and its proof is based on the Arzel\'a-Ascoli Theorem and on the estimation of the modulus of continuity. We are unable to carry out this argument in the general case of LCA groups.
Thus, we will prove the LCA group counterpart in a slightly different way, involving some basic notions and properties from the theory of topological vector spaces, which are given in the forthcoming paragraphs.

Let $A\subseteq X$ be any subset of a locally convex (Hausdorff) topological vector space $X$.
Then $A$ is called \emph{totally bounded}, whenever
for every neighbourhood $V$ of the origin,
there is a finite subset $S\subseteq A$
such that $A$ is contained in $S+V$.
Obviously, this requirement can be
equivalently assumed only for
neighbourhoods belonging to a given
neighbourhood base of $X$.

A topological vector space $X$ is called \emph{complete} if every Cauchy net has a limit in $X$.
Further for any locally convex topological
space $X$, there is a unique
(up to a linear homeomorphism) pair
$(\widetilde{X},j)$ of a complete space $\widetilde{X}$ and a linear homeomorphic
embedding $j: X \to \widetilde{X}$ such
that $j(X)$ is dense in $\widetilde{X}$.

If $G$ is an LCA group, then the
relative compactness can be verified by using the following result from the  theory of locally convex
spaces (see, for instance \cite[3.5.1. Theorem]{Jarchov}).

\begin{lemma} \label{L:equivalentstates}
For every subset $E$ of a locally convex topological vector space $X$, the following are equivalent.
\begin{itemize}
    \item[1)] $E$ is totally bounded.
    \item[2)] $E$ is relatively compact in the completion of $X$.
    \item[3)] Every sequence of $E$ has a cluster point in the completion of $X$.
\end{itemize}
\end{lemma}

\begin{proof}[Proof of Lemma~\ref{L:precompact}]
According to Lemma \ref{L:equivalentstates} we are going to show that $\mathcal{G}(W,Q)_G$ is totally bounded. First note that in the space $C(G)$ (equipped with the compact convergence topology) for any compact set $K\subseteq G$ and any $\varepsilon > 0$, the $U(f;K,\varepsilon)$-neighbourhood of the function $f\in C(G)$ is defined as
\[
U(f;K,\varepsilon) = \{h \in C(G) ~ : ~ \|h-f\|_{C(K)} < \varepsilon \}.
\]
This forms the defining neighbourhood base for compact convergence on $C(G)$.

So our aim is to show that for any
$\varepsilon >0$ and any compact set
$K\subseteq G$, there exists a finite set
$\{ f_1, \ldots , f_n \}\subseteq \mathcal{G}(W,Q)_G$ such that
\[
\mathcal{G}(W,Q)_G \subseteq \bigcup_{j=1}^n U(f_j; K,\varepsilon).
\]

As by assumption $Q\subseteq \widehat{G}$ is compact in the compact convergence topology, $Q$ is totally bounded as well. It means that there exists a finite set $\{\chi_1, \ldots, \chi_n\} \subseteq Q$ such that for every $\gamma \in Q$ we get that $\|\gamma - \chi_j\|_{C(K)} < \varepsilon$ for some $j\in \{1, \ldots, n \}$.
Via the disjointization procedure
\[
\begin{gathered}
Q_1:=U(\chi_1;K,\varepsilon) \cap Q  \\
Q_2:=U(\chi_2;K,\varepsilon) \cap (Q \setminus Q_1) \\
\vdots \\
Q_n:=U(\chi_n;K,\varepsilon) \cap (Q \setminus (Q_1 \cup \ldots \cup Q_{n-1}))
\end{gathered}
\]
we obtain a partition $\{Q_1, \ldots , Q_n \}$ of $Q$ such that every $Q_j$ ($j=1,\ldots,n$) is a Borel sets with compact closure, and for any $\gamma \in Q_j$ we have $\|\gamma-\chi_j\|_{C(K)}<\varepsilon$.

Next choose an element $f\in \mathcal{G}(W,Q)_G$ and define
\[
F(g) := \sum_{j=1}^n \chi_j(g)\int_{Q_j}\widehat{f}(\chi)d\lambda_{\widehat{G}}(\chi) \equiv \sum_{j=1}^n c_j(f) \chi_j(g)
\]
with
$c_j(f):=\int_{Q_j} \widehat{f}(\chi) d\lambda_{\widehat{G}}(\chi)$
where $c_j(f)\geq 0$ because of the positive definiteness of $f$, and 
\[
\sum_{j=1}^n c_j(f) = \int_{\widehat{G}} \widehat{f}(\chi) d\lambda_{\widehat{G}}(\chi) = f(0) = 1 .
\]
Note that $f \in \mathcal{G}(W,Q)_G\subseteq L^1(G)$
and $\supp{\widehat{f}}\subseteq Q$
implies $\widehat{f}\in L^p\left(\widehat{G}\right) \quad
(1\leq p \leq \infty$). By the Inversion Theorem, we get for $g\in K$ in view of $\widehat{f} \geq 0$ that
\[
|f(g) - F(g)| = \left|  \sum_{j=1}^n \int_{Q_j} \left(\chi(g)-\chi_j(g)\right)\widehat{f}(\chi) d \lambda_{\widehat{G}}(\chi) \right| \leq  \varepsilon \cdot f(0) = \varepsilon
\]
and so $\|f-F\|_{C(K)} \leq \varepsilon$.
Let $m>n/\varepsilon$ be an integer, and define $d_j(f):=[m \cdot c_j(f)]/m$.
Then we have
\[
\norm{  \sum_{j=1}^n c_j(f)\chi_j-\sum_{j=1}^n d_j(f)\chi_j }_{C(K)} < \sum_{j=1}^n \frac{1}{m}=\frac{n}{m} < \varepsilon.
\]
It follows that
\[
\norm{ f - \sum_{j=1}^n d_j(f) \chi_j}_{C(K)} \leq \|f-F\|_{C(K)} + \norm{ F - \sum_{j=1}^n d_j(f) \chi_j  }_{C(K)} < 2\varepsilon.
\]
For any choice of the function $f\in \mathcal{G}(W,Q)_G$ one has $d_j(f) \in \{ 0, 1/m, \ldots, 1 \}$, whence the set
\begin{equation} \label{eq:linearcomb}
    \left\{ \sum_{j=1}^m r_j \chi_j ~ : ~ r_j \in \{ 0, 1/m, \ldots, 1 \}   \right\}
\end{equation}
forms a finite $2\varepsilon$-net for $\mathcal{G}(W,Q)_G$ on K with respect to the compact convergence topology. This holds for any base neighbourhood of the form $U\left({\bf 0};K,2\varepsilon\right)$.
Therefore, we found a finite net
(\ref{eq:linearcomb}) such that the
respective translates of
$U\left({\bf 0};K,2\varepsilon\right)$
cover $\mathcal{G}(W,Q)_G$, so the function set
$\mathcal{G}(W,Q)_G$ is totally bounded.
\end{proof}

In the sequel we make crucial use of the following selection lemma.

\begin{lemma} \label{L:fn_conv}
Suppose that $G$ is $\sigma$-compact.
Let $(f_n)$ be a sequence in $\mathcal{G}(W,Q)_G$.
Then there exists a subsequence of $(f_n)$ which converges to a function $f\in\mathcal{G}(W,Q)_G$ uniformly on every compact set and also in the weak-$^*$ sense. Moreover, we have the inequality
\begin{equation} \label{eq:int_inequality}
\int_G f d\lambda_G \geq \limsup_{n\to\infty}\int_G f_n d\lambda_G.
\end{equation}
\end{lemma}

\begin{proof}[Proof of Lemma~\ref{L:fn_conv}]
In the first part of the proof, we use the arguments given in \cite{Berdysheva}.
Let $(f_n)$ be a sequence in 
$\mathcal{G}(W,Q)_G$. Using Lemmata
 \ref{L:precompact}, \ref{L:equivalentstates} 
 and the completeness of $C(G)$ with 
 respect to the compact convergence topology, we conclude that there exists a subsequence of $(f_n)$ which tends to some $f\in C(G)$ uniformly on every compact set, and thus also in the pointwise sense.
Without loss of generality we may and do assume that $(f_n)$ itself converges to $f$.

Next we intend to show that $f\in \mathcal{G}(W,Q)_G$.
Since the pointwise limit of positive definite functions is likewise positive definite, it follows that $f \gg 0 $ holds.
As $W \subseteq G$ is closed, we clearly have $\supp f_+ \subseteq \overline{W} = W$, $f(0)=1$ and $|f|\leq 1$.

Now we are concerned with verifying that $f$ belongs to $L^1(G)$.
Writing $f=f_+ - f_-$ and, in a similar fashion, $f_n=(f_n)_+-(f_n)_-$ one has $(f_n)_{\pm} \to f_{\pm}$ in the pointwise sense. An application of Fatou's lemma gives us that
\begin{equation}\label{eq:fnminus}
\int_G  f_- d\lambda_G \le \liminf_{n\to\infty} \int_G (f_n)_- d\lambda_G .
\end{equation}
For the positive parts note that $(f_n)_+$
and $f_+$ are all supported in $W$,
and $|(f_n)_+|\le (f_n)_+(0)=1$,
all the functions $f_n$ belonging to $\mathcal{G}(W,Q)_G$.
That is, $(f_n)_+ \le \mathbb{1}_W$,
which is integrable because $W$ has
finite Haar measure. Therefore, the
Lebesgue Dominated Convergence Theorem yields
\begin{equation}\label{eq:fnplus}
\int_G  f_+ d\lambda_G = \lim_{n\to\infty} \int_G (f_n)_+ d\lambda_G.
\end{equation}
Note that then
\[
\begin{gathered}
\int_G |f| d\lambda_G =\int_G f_+ d\lambda_G + \int_G f_- d\lambda_G
 \le  \lim_{n\to \infty} \int_G (f_n)_+ d\lambda_G ~ +\\ + ~ \liminf_{n\to \infty} \int_G (f_n)_- d\lambda_G 
\le 2 \lim_{n\to \infty} \int_G (f_n)_+ d\lambda_G \le 2\lambda_G(W)
\end{gathered}
\]
because
\[
\int_G (f_n)_- d\lambda_G = \int_G \left(  (f_n)_+ - f_n \right)d\lambda_G  \le \int_G (f_n)_+ d\lambda_G
\]
for each $n$, for $\int_G f_n d\lambda_G  \ge 0$ in view of $f_n \gg 0$. Therefore, we have also proved $f \in L^1(G)$, that is, also $f\in C(G) \cap L^1(G) \cap L^\infty(G)$, whence it belongs to $L^2(G)$. In particular, $\widehat{f}$ does exist, is continuous and belongs to $L^2(\widehat{G})$.

It remains to show that $\supp \widehat{f} \subseteq Q$. 
Here we need to argue in a different way than \cite{Berdysheva} does.
Clearly, the linear functional
\[
\psi_{\rho}(\varphi):=\int_G  \varphi \overline{\rho} d \lambda_G,\qquad \textrm{for }\varphi \in L^1(G)
\]
belongs to the unit ball in the dual space of $L^1(G)$ for $\rho=f_n$ or $\rho=f$.

Using that $G$ is $\sigma$-compact, we have that the space $L^1(G)$ is WCG, moreover, $\left( L^1(G)\right)^*=L^{\infty}(G)$.
Thus there is a subsequence of $(f_n)$ (supposed to be itself $(f_n)$ again) which converges to some $f_0\in L^{\infty}(G)$ in the weak-$^*$ sense. It is not difficult to verify that $f_0$ must coincide with the locally uniform limit function $f$, that is, we have
\begin{equation} \label{weaklimit}
\int_G f_n \varphi d\lambda_G \longrightarrow \int_G f \varphi d\lambda_G, \qquad \textrm{for } \varphi \in L^1(G).
\end{equation}

Take any $\gamma \in \widehat{G}\setminus Q$, and a small symmetric neighbourhood $B$ of the unit element ${\bf 1}$ of $\widehat{G}$ with compact closure satisfying $\gamma {BB} \cap Q = \emptyset$. Define the functions $\theta_{\gamma}(\chi):=(\mathbb{1}_B \star \mathbb{1}_B)(\chi\gamma^{-1})$ and $\theta(\chi):=\theta_{\bf 1}(\chi)$.
Note that $\theta$ is compactly supported and $\theta({\bf 1})=(\mathbb{1}_B \star \mathbb{1}_B)({\bf 1})=\lambda_{\widehat{G}}(B)$.

Since $B$ is symmetric, we get $\mathcal{F}^*(\theta)=\left| \widehat{\mathbb{1}_B} \right|^2$, by elementary properties of the $L^2$-Fourier transform \cite[Section 1.6]{Rudin}.
This immediately yields that $\mathcal{F}^*(\theta)\in L^1(G)$. Indeed, one has
\[
\begin{gathered}
\| \mathcal{F}^*(\theta) \|_1 =
\norm{ \widehat{\mathbb{1}_B} }_2^2 = \norm{ \mathbb{1}_B }_2^2 = \lambda_{\widehat{G}}(B) < +\infty
\end{gathered}
\]
because $B$ has compact closure and the Haar measure is locally finite.
Thus, we also have
\begin{multline*}
h(g):=\mathcal{F}^*\left( \theta_\gamma \right)(g)= \gamma(g)\cdot \mathcal{F}^*( \mathbb{1}_B \star \mathbb{1}_B )(g) \\ =
\gamma(g)\cdot |\mathcal{F}^*(\mathbb{1}_B)(g)|^2 =
\gamma(g) \cdot \left| \widehat{\mathbb{1}_B}(\delta_g)\right|^2
\end{multline*}
where the last equality follows from the symmetry of $B$.
We see that $h\in L^1(G)$. So we can take $k:=f\star h \in L^1(G)$ for which we clearly have $\mathcal{F}(k)=\widehat{f}\theta_{\gamma}$.

Hence we conclude from \eqref{weaklimit} via the Plancherel Theorem that
\begin{multline*}
k(s)=\lim_{n\to \infty} \int_G f_n(g) h(s-g) d\lambda_G(g)
\\ =
\lim_{n\to \infty} \int_{\widehat{G}}
\widehat{f_n}(\chi)\delta_s(\chi)\theta_{\gamma}(\chi)d\lambda_{\widehat{G}}(\chi) = 0
\end{multline*}
in view of $\supp \widehat{f_n} \subseteq Q$ and $\{ \theta_{\gamma} \neq 0 \} \cap Q = \gamma {BB} \cap Q=\emptyset$.
Therefore, $k(s)=0$ holds for all $s\in G$. Taking Fourier transform gives $\widehat{f}\theta_{\gamma} \equiv 0$, in particular,
\[
0=\widehat{f}(\gamma)\theta_{\gamma}(\gamma)=\widehat{f}(\gamma)\theta({\bf1})=\widehat{f}(\gamma)\lambda_{\widehat{G}}(B).
\]
Thus, at any point $\gamma$ outside the set
$Q$ the function $\widehat{f}$
vanishes. It follows that
$\supp \widehat{f} \subseteq Q$
and so $f\in \mathcal{G}(W,Q)_G$ as wanted.

By substracting \eqref{eq:fnminus} from \eqref{eq:fnplus}, we immediately get the last statement of the Lemma:

\begin{multline*}
\int_G f d\lambda_G \geq \lim_{n \to \infty} \int_G (f_n)_+ d\lambda_G - \liminf_{n \to \infty} \int_g (f_n)_- d\lambda_G 
\\\geq
\limsup_{n\to\infty} \int_G \left( (f_n)_+ - (f_n)_- \right)=\limsup_{n\to\infty} \int_G f_n d\lambda_G.
\end{multline*}

\end{proof}

From the definition it is clear that the restriction of a positive definite function to a subgroup remains positive definite on the subgroup as well. For the following fact, the reader can consult with \cite[(32.43) (a)]{HewittRossII}.
\begin{lemma} \label{L:extension}
Let $H \leq G$ be a closed subgroup of $G$.
If the function $f:H\to \mathbb{C}$ is continuous and positive definite, then so is its trivial extension $\widetilde{f}:G\to \mathbb{C}$ defined by
    \begin{equation} \label{eq:trivi-extension}
    \widetilde{f}(g)=
    \begin{cases}
    f(g) &\quad\text{if } g\in H; \\
    0 &\quad\text{otherwise.}   \\
    \end{cases}
    \end{equation}
\end{lemma}

%In the proof of the main result we consider characters on subgroups and quotients groups as well. The applied results (see \cite[Theorem 2.1.2 and Theorem 2.1.4]{Rudin}) are summarized in the forthcoming lemma. Recall that if $H \leq G$ is a subgroup, then the annihilator $N$ of $H$ is the set
%\[
%N=\{\chi \in \widehat{G} ~:~ \chi(g) = 1 ~ \text{for every} ~ g\in H \}.
%\]

%\begin{lemma}
%     Let $H\leq G$ be a closed subgroup. 
%     \begin{itemize}
%         \item[1)] $\widehat{G/H}$ isomorphically homeomorphic to the annihilator $N$ of $H$;
%         \item[2)] Every character of $H$ extends to a character of $G$.
%     \end{itemize}
%\end{lemma}

Now we are in a position to prove Theorem \ref{T:main}.
Our strategy is the following.
First we prove the theorem in the case where the underlying group is $\sigma$-compact, and then we reduce the general case after some technical preparation to the $\sigma$-compact one.

\begin{proof}[Proof of Theorem~\ref{T:main}]
At first we suppose that $G$ is $\sigma$-compact. 
Using the definition of $\sup$, there is an extremal sequence $(f_n) \subseteq \mathcal{G}(W,Q)_G$ such that
\begin{equation} \label{eq:sequenceoffunctions}
    \int_G f_n d\lambda_G > \mathcal{D}(W,Q)_G-\frac{1}{n}, \quad n\in\mathbb{N}^+.
\end{equation}
According to Lemma \ref{L:fn_conv}, there exists a limit function $f\in \mathcal{G}(W,Q)_G$ such that a subsequence of $(f_n)$ converges to $f$. Without loss of generality, we may and do suppose that $(f_n)$ converges to $f$.
We show that $f$ is the extremal function in $\mathcal{G}(W,Q)_G$, that is,
$\int_G f d\lambda_G=\mathcal{D}(W,Q)_G$.
By using the definition of the extremal constant,
 inequality \eqref{eq:int_inequality} and definition \eqref{eq:sequenceoffunctions}, we get

\begin{align*} 
\mathcal{D}(W,Q)_G \geq \int_G f d\lambda_G \geq
\limsup_{n\to \infty} \int_G f_n d\lambda \geq \mathcal{D}(W,Q)_G 
\end{align*}
and thus we have equality everywhere in the last displayed chain of inequalities. This completes the proof when $G$ is $\sigma$-compact.

Assume now that $G$ is not $\sigma$-compact. 
Let $G_0$ the open, $\sigma$-compact subgroup which is generated by $W$, that is, $$V:=W-W, \qquad G_0:=\bigcup\limits_{n\in \mathbb{N}}nV.$$
%The characteristic
%function $\mathbb{1}_W$ of $W$ is integrable
%on $G$ because $W$ is closed and of finite
%Haar measure.
%We know \cite[Corollary 1.3.5 (d)]{Deitmar} that the support of every
%integrable function is
%contained in a $\sigma$-compact open
%subgroup of $G$.
%Moreover, the union of countable many open $\sigma$-compact subgroups generates an open $\sigma$-compact subgroup \cite[Proposition 1.2.1 (c)]{Deitmar}.
%It follows that the union of the supports of
%countably many integrable functions
%$(h_n \in L^1(G), n\in \mathbb{N})$
%is also contained in a
%$\sigma$-compact open subgroup
%of $G$.
%Denote by $G_0$ the $\sigma$-compact open subgroup in which $W$ is contained.
Then $G_0$ is an LCA group and a Haar measure on $G_0$ is given by $\lambda_{G_0}:=\lambda_G|_{G_0}$. 
%We remark that the definition of $G_0$ implies that $f_n$ defined on $G$ vanishes everywhere outside $G_0$.
Define the sets $Q^*,Q_0$ as
\begin{align*}
    Q^*:=\left\{ \gamma \in \widehat{G_0}: \text{all the extensions of }\gamma \text{ to }G \text{ lie in }Q  \right\} \\  
    Q_0:=\left\{ \gamma \in \widehat{G_0}: \ \exists \chi \in Q \ \text{such that} \ \chi|_{G_0}=\gamma\right\}.
\end{align*}

\begin{claim}
The set $Q^*\subseteq \widehat{G_0}$ is compact.
\end{claim}

Clearly, we have $Q^* \subseteq Q_0$.
The set $Q_0$ is the image of the compact set $Q$ under the restriction map $$
\Phi:\widehat{G} \to \widehat{G_0}, \quad \chi \mapsto\chi|_{G_0}.
$$
Since $G_0$ is open, according to \cite[(24.5) Lemma]{HewittRossI} $\Phi$ is continuous.
%Note that the annihilator 
%\[
%G_0^{\perp}=\{\chi \in \widehat{G} ~:~ \chi(g)=1 \text{ for every } g\in G_0 \}
%\]
%of $G_0$ satisfies $\widehat{G}/G_0^{\perp} \cong \widehat{G_0}$ (isomorphically homeomorphic), by \cite[2.1.2. Theorem]{Rudin}.
%Identifying $\widehat{G_0}$ with $\widehat{G}/G_0^{\perp}$, the restriction map $\Phi$ is just the canonical projection of $\widehat{G}$ onto $\widehat{G}/G_0^{\perp}$ which is continuous, whence $Q_0$ is compact.
So $Q^*$ is compact if and only if it is a closed subset of the compact set $Q_0$.
We can write the complement of $Q^*$ as
\begin{align*}
(Q^*)^c =& \{ \xi \in \widehat{G_0} ~ : ~ \exists \chi \in \widehat{G}, ~ \chi|_{G_0}=\xi, ~ \chi \notin Q \}  \\
=& \bigcup_{\chi \in \widehat{G} \setminus Q} \left\{ \chi|_{G_0} \right\} = \Phi\left( \widehat{G} \setminus Q \right)    
\end{align*} 
where the latter set is open because $\Phi$ is an open mapping, again by \cite[(24.5) Lemma]{HewittRossI}.

\medskip

Similarly to \eqref{G(W,Q)} and \eqref{D(W,Q)}
we consider the function class
$\mathcal{G}(W,Q^*)_{G_0}$
and the extremal quantity
$\mathcal{D}(W,Q^*)_{G_0}$.

\begin{claim}
We have $h^0\in \mathcal{G}(W,Q^*)_{G_0}$ if and only if for its extension ${h}$ we have ${h}\in \mathcal{G}(W,Q)_G$.   
\end{claim}

First assume that $h\in \mathcal{G}(W,Q)_G$. Since further properties of $h^0:=h|_{G_0}$ are inherited to that of $h$, we intend to show that $\supp \widehat{h^0} \subseteq Q^*$.
Choose a $\gamma \in \widehat{G_0}$ for which $\widehat{h^0}(\gamma)\neq 0$.
Let $\chi\in \widehat{G}$ be any extension of $\gamma$ such that $\chi|_{G_0}=\gamma$.
As $h\in L^1(G)$ with $\supp{h} \subseteq G_0$, there holds the computation
\begin{equation} \label{fouriertr}
    \widehat{h^0}(\gamma)=\int_{G_0} h^0(g)\overline{\gamma(g)} d\lambda_{G_0}(g)=
    \int_G h(g) \overline{\chi(g)} d\lambda_G(g)=\widehat{h}(\chi),
\end{equation}
whence $\widehat{h}(\chi)\neq 0$ holds.
It follows that every extension $\chi \in \widehat{G}$ of $\gamma$ lies in $Q$, in other words $\gamma \in Q^*$. Thus, $\supp \widehat{h^0} \subseteq Q^*$, as wanted.

%From this we conclude that if $\widehat{h}(\chi) \neq 0$, then $\widehat{h}(\chi') \neq 0$ whenever $\chi|_{G_0}=\chi'|_{G_0}=\gamma$. 
%So we find that if $\widehat{h}(\gamma)\neq 0$, then every extension of $\gamma$ is in $Q$, in other words $\gamma \in Q^*$, and thus 
%$ \supp \widehat{h^0} \subseteq Q^*$.
%So if $h \in \mathcal{G}(W,Q)_{G}$, then we conclude that $h^0 \in \mathcal{G}(W,Q^*)_{G_0}$ holds.

To see the converse, again from the computation \eqref{fouriertr} we see that $\widehat{h}(\chi)\neq 0$ implies $\widehat{h^0}(\gamma)\neq 0$ whenever $\gamma$ is the restriction of $\chi$. By assumption and the definition of the set $Q^*$, the character $\chi$ lies in $Q$. So $\supp\widehat{h}\subseteq Q$, and thus $h\in \mathcal{G}(W,Q)_G$.  

\begin{claim}
We have $\mathcal{D}(W,Q)_G = \mathcal{D}(W,Q^*)_{G_0}$.
\end{claim}

The inequality $\mathcal{D}(W,Q)_G \leq \mathcal{D}(W,Q^*)_{G_0}$ is in fact easy to verify. Indeed, by the $\mathcal{D}(W,Q)_{G}$-extremality of the sequence $(f_n)$ and $f_n^0:=f_n|_{G_0} \in \mathcal{G}(W,Q^*)_{G_0}$ for every $n\in \mathbb{N}$, 
we get
\[
\mathcal{D}(W,Q^*)_{G_0} \geq \int_{G_0} f_n^0 d\lambda_{G_0} = \int_G f_n d\lambda_G > \mathcal{D}(W,Q)_{G} - \frac{1}{n}
\]
for every $n\in \mathbb{N}^{+}$, as wanted.

To see the converse, consider an extremal sequence $(f_n^0) \subseteq  \mathcal{G}(W,Q^*)_{G_0}$ on $G_0$ and extend it in the trivial way
to a sequence $(\widetilde{f_n})$ on $G$.
Then in virtue of \eqref{fouriertr} $\widehat{\widetilde{f_n}}(\chi)\neq 0$ implies $\widehat{f_n^0}(\gamma)\neq 0$ where $\gamma=\chi|_{G_0}$. 
The latter condition gives us that $\gamma \in Q^*$, so it follows directly that $\supp \widehat{\widetilde{f_n}} \subseteq Q$.
Hence $\widetilde{f_n} \in \mathcal{G}(W,Q)_{G}$ and thus
\[
\mathcal{D}(W,Q)_G \geq \int_G \widetilde{f_n} d\lambda_G = \int_{G_0} f_n^0 d\lambda_{G_0}> \mathcal{D}(W,Q^*)_{G_0}-\frac{1}{n}
\]
for every integer $n\geq 1$ which implies $\mathcal{D}(W,Q)_G\geq \mathcal{D}(W,Q^*)_{G_0}$. 

\medskip

Now we can finish the proof of Theorem~\ref{T:main} quite easily.
To do so, consider a $\mathcal{D}(W,Q^*)_{G_0}$-extremal sequence $(f_n^0)$ on $G_0$.
Then according to Lemma~\ref{L:fn_conv} there is a subsequence of $(f_n^0)$ which tends to a function $f_0 \in \mathcal{G}(W,Q^*)_{G_0}$, and we have
\[
\int_{G_0} f_0  d\lambda_{G_0} = \mathcal{D}(W,Q^*)_{G_0}= \mathcal{D}(W,Q)_G.
\]
For the trivial extension $\widetilde{f}$ of $f_0$, one has
\[
\int_{G} \widetilde{f}  d\lambda_{G} = \int_{G_0} f_0  d\lambda_{G_0} = \mathcal{D}(W,Q)_G.
\]
Since $\widetilde{f} \in \mathcal{G}(W,Q)_G$, the last displayed equality shows that the function $\widetilde{f}$ is $\mathcal{D}(W,Q)_G$-extremal.  
\end{proof}

\begin{Remark}
It is apparent from the construction presented in the last part of the proof of Theorem~\ref{T:main} that the extremal function can be chosen to be supported in the open $\sigma$-compact subgroup of $G$ generated by $W$. 
\end{Remark}

\bigskip

\textbf{Acknowledgement.}
This research was partially supported by the DAAD-Tempus PPP Grant 57448965 titled ,,Harmonic Analysis and Extremal Problems''.

Ga\'al was supported by the National Research, Development and Innovation Office -- NKFIH Reg. No.'s K-115383 and K-128972, and also by the Ministry for Innovation and Technology, Hungary throughout Grant TUDFO/47138-1/2019-ITM.

Zsuzsanna Nagy-Csiha was supported by the ÚNKP-19-3 New National Excellence Program of the Ministry for Innovation and Technology.

The authors gratefully acknowledge their sincere thank to Szil\'ard Gy. R\'ev\'esz for great discussions and encouragement.
They also thank to Elena Berdysheva for several useful comments and suggestions, and for the reference \cite{CG}. The help of Dávid Kunszenti-Kovács, who gave a comment on the earlier version of the paper, is also acknowledged.  

\bibliographystyle{amsplain}

\end{document}